\UseRawInputEncoding

\documentclass[12pt]{article}
\usepackage{amsmath, amssymb, amsthm, a4wide}
\usepackage{geometry}
\usepackage{fancyhdr}
\usepackage{enumerate}
\usepackage{multirow}
\newtheorem{theorem}{Theorem}[section]

\newtheorem{corollary}[theorem]{Corollary}
\newtheorem{definition}[theorem]{Definition}
\newtheorem{proposition}[theorem]{Proposition}
\newtheorem{example}[theorem]{Example}
\newtheorem{remark}[theorem]{Remark}

\numberwithin{equation}{section}

\allowdisplaybreaks[4]
\newcounter{newlist}

\newcounter{nnnewlist}

\newcounter{nelist}

\def\tX{\tilde{X}}

\def\EE{\mathbb{E}^{\cp}}

\def\lt{\left}
\def\rt{\right}

\def\lu{\underline{\mu}}
\def\ou{\overline{\mu}}
\def\ls{\underline{\sigma}}
\def\os{\overline{\sigma}}

\def\e{\mathbb{E}}

\def\eg{{\mathbb{E}_G}}
\def\vp{\varphi}
\def\ve{\varepsilon}
\def\e{\mathbb{E}}

\def\br{\mathbb{R}}
\def\cs{\mathcal{S}}
\def\cp{\mathcal{P}}
\def\cf{\mathcal{F}}
\def\cb{\mathcal{B}}

\def\cf{\mathcal{F}}
\def\cn{\mathcal{N}}

\def\cg{\mathcal{G}}

\def\lu{\underline{\mu}}
\def\ou{\overline{\mu}}
\def\ls{\underline{\sigma}}
\def\os{\overline{\sigma}}

\def\be{\hat{\mathbb{E}}}

\def\bn{\mathbb{N}}

\def\pe{{\mathbb{E}^{\cp}}}

\def\eg{\mathbb{E}_G}

\def\cn{\mathcal{N}}

\def\vp{\varphi}
\def\ve{\varepsilon}

\begin{document}

\title{Functional central limit theorem with mean-uncertainty under sublinear expectation}

 \author{Xiaofan GUO\ \ \ Xinpeng LI \thanks{
Corresponding author:  Xinpeng LI, Email: lixinpeng@sdu.edu.cn}\\
\\ Reserach Center for Mathematics and Interdisciplinary Sciences\\Shandong University\\266237, Qingdao, China
}
\date{ }

\maketitle

\noindent\textbf{Abstract.}
In this paper, we introduce a fundamental  model for independent and identically distributed sequence with model uncertainty on the canonical space $(\mathbb{R}^\bn,\mathcal{B}(\br^\bn))$ via probability kernels. Thanks to the well-defined upper and lower variances, we obtain a new functional central limit theorem with mean-uncertainty on the canonical space by the method based on the martingale central limit theorem and stability of stochastic integral in the classical probability theory. Then we extend it to the general sublinear expectation space through a new representation theorem. Our results generalize Peng's central limit theorem with zero-mean to the case of mean-uncertainty and provides a purely probabilistic proof instead of the existing nonlinear partial differential equation approach. As an application, we consider the two-armed bandit problem and generalize the corresponding central limit theorem from the case of mean-certainty to mean-uncertainty.\\



\section{Introduction}

The notions of independence and identical distribution are crucial in the probability and statistics.
In classical probability theory, we usually resort to product measure $\mu=\otimes_{i=1}^\infty\mu_i$ to construct the independent canonical sequence $\{X_i\}_{i\in\bn}$ on the canonical space $(\br^\bn,\cb(\br^\bn))$ with prescribed marginal laws $\{\mu_i\}_{i\in\bn}$. Generally, any joint law $\mu$ on $(\br^\bn,\cb(\br^\bn))$ can be decomposed in terms of its probability kernels in the following form (see Yan \cite{yan}): $\forall n\in\bn$, $ \forall A\in\cb(\br^n)$,
$$\mu(A\times\br^{\bn-n})=\int_\br\mu_1(dx_1)\int_{\br}\kappa_{2}(x_1,dx_2)\cdots\int_\br I_{A}\kappa_n(x_1,\cdots,x_{n-1},dx_{n}).$$
In particular, if the probability kernel $\kappa_i(x_1,\cdots,x_{i-1},\cdot)$ dose not rely on  $(x_1,\cdots,x_{i-1})$ for all $i\geq 2$, the canonical sequence $\{X_i\}_{i\in\bn}$ is independent under $\mu$. If further assume that $\kappa_i(x_1,\cdots,x_{i-1},\cdot)\equiv\mu_1(\cdot)$, then $\{X_i\}_{i\in\bn}$ is independent and identically distributed (i.i.d.). This formulation inspires us to introduce a fundamental model on canonical space $(\br^\bn,\cb(\br^\bn))$ via probability kernels such that the canonical sequence $\{X_i\}_{i\in\bn}$ is independent and identically distributed with model uncertainty as described in Peng \cite{P2019}, i.e., we construct a set of probability measures generated by probability kernels. Such formulation provides a new interpretation of i.i.d. sequence on sublinear expectation space.

The i.i.d. assumption is usually set in the central limit theorem (CLT). The new notion of i.i.d. sequence on the sublinear expectation space $(\Omega,\mathcal{H},\be)$ was initially introduced in Peng \cite{P2019b} and  the corresponding CLT with zero-mean was established in Peng \cite{P08,pengsur,P09,P2019,P2019b}, known as Peng's CLT:

For an i.i.d. sequence $\{X_i\}_{i\in\bn}$ with $\be[X_1]=\be[-X_1]=0$ and
\begin{equation}\label{uc}
\lim_{\lambda\rightarrow\infty}\mathbb{\hat{E}}[(|X_1|^2-\lambda)^+]=0,
\end{equation}
we have
$$\lim_{n\rightarrow\infty}\mathbb{\hat{E}}\lt[\vp\lt(\frac{\sum_{i=1}^nX_i}{\sqrt{n}}\rt)\rt]=\eg[\vp(\xi)], \ \ \forall \vp\in C_{b.Lip}(\br),$$
where $\eg$ is the $G$-expectation corresponding to a $G$-normally distributed random variable $\xi$ with zero-mean and upper and lower variances $\be[X_1^2]$ and $-\be[-X_1^2]$, which is characterized by the so-called $G$-heat equation. The corresponding proof adopted the partial differential equation (PDE) approach which is mainly based on the deep result of the interior regularity for PDE (see Krylov \cite{Kry2}). We emphasize that the regularity of $\be$ (see Definition \ref{reg}) is not necessary here. {For instance, in the case of only finitely additive $\be$, the above CLT still holds, but the classical probability theory is not applicable since there does not exist any probability measure such that underlying random variables are i.i.d. (see Example \ref{ex1}).}

After Peng established CLT on sublinear expectation space, Zhang \cite{Zhang} obtained the sufficient and necessary conditions of CLT for i.i.d. random variables, which weakened the condition (\ref{uc}) to $\lim_{\lambda\rightarrow\infty}\lambda\mathbb{V}(|X_1|^2\geq\lambda)=0$, where $\mathbb{V}$ is the capacity introduced by $\mathbb{\hat{E}}$. Peng \cite{P09} proposed a new proof of CLT by weakly compact method but involved PDE approach to characterize the $G$-normal distribution. Song \cite{song2} provided an estimate of the convergence rate for CLT by Stein's method, and Krylov \cite{Kry1} gave error estimates in CLT for not necessarily non-degenerate case by the finite-difference approximations for Bellman's equation. Besides, convergence to the $G$-normal distribution also occurs for non-independent random variables or non-identical distributions (see Li \cite{Li2015}, Li and Shi \cite{LS}, Zhang \cite{zhang2020}). Zhang \cite{Z2015} also considered the functional CLT based on Peng's CLT in \cite{P2019}. We note that almost all papers on CLT under sublinear expectation assume that the underlying random variables are zero-mean and most of proofs rely on PDE approach.

{An interesting problem is that how about CLT for the random variables with mean-uncertainty? Recently, Chen and Epstein \cite{CE} proved a CLT for random variables with mean-uncertainty and unambiguous conditional variance, where the limit is defined by a backward stochastic differential equation (see Pardoux and Peng \cite{PP}, Peng \cite{P97}).  One important assumption in \cite{CE} is that all measures in the given set $\cp$  are equivalent on the filtration, while such assumption is not necessary in our paper. Fang et al. \cite{FPSS} also obtained a CLT with mean-uncertainty convergent to the classical normal distribution under additional assumptions.}

Motivated by the weak  approximation of $G$-expectations introduced by Dolinsky et al. \cite{DNS}, we establish the functional CLT on the canonical space $(\br^\bn,\mathcal{B}(\br^\bn), \cp)$, where $\cp$ is a set of probability measures introduced via probability kernels with model uncertainty. The proof relies on the martingale central limit theorem and stability of stochastic integral in the classical probability theory.
The determination of parameters in our CLTs are due to the well-defined upper and lower variances. It is easy to extend functional CLT from the canonical space to the general sublinear expectation space by the representation theorem. Thus we provide a new and purely probabilistic proof of Peng's CLT, then generalize it to the mean-uncertainty case. We note that the corresponding limit distribution is still $G$-normally distributed. This new CLT illustrates the broad applicability of $G$-normal distribution for the situations with mean-uncertainty. For instance, Chen et al. \cite{CEZ} studies the two-armed bandit problem from the perspective of sublinear expectations and obtains a CLT in the case of mean-certainty. Generalizing such result to the case of mean-uncertainty, we establish a functional CLT for two-armed bandit taking $G$-normal distribution as the limitation, as an application of our main result.

This paper also provides a new methodology to study limit theorems on general sublinear expectation space. We firstly establish desired limit theorems on the canonical space $(\mathbb{R}^\bn,\mathcal{B}(\br^\bn),\cp)$ applying the classical martingale limit theorems. Thanks to a new representation theorem established in this paper, we could extend them to the general sublinear expectation space $(\Omega,\mathcal{H},\be)$.

The remainder of this paper is organized as follows. Section 2 describes the fundamental model for i.i.d. sequence on canonical space $(\br^\bn,\cb(\br^\bn))$ introduced by probability kernels with model uncertainty. The corresponding functional CLT with mean-uncertainty are established in Section 3. In Section 4, we extend functional CLT from canonical space to the general sublinear expectation space. Some applications and examples are provided in Section 5. The law of large numbers, which is used in the proof of functional CLT, is given in Appendix.

\section{Fundamental model on canonical space}

\subsection{Independent sequence on canonical space}
Let $(\br^\bn,\cb(\br^\bn))$ be the canonical space with the metric $d(\boldsymbol{x}, \boldsymbol{y})=\sum_{i=1}^{\infty}\left(\left|x_i-y_i\right| \wedge 1\right) / 2^i$. Suppose that $\{X_i\}_{i\in\bn}$ is the sequence of canonical random variables defined by $X_i(\omega)=\omega_i$ for $\omega=(\omega_1,\cdots,\omega_n,\cdots)\in\br^\bn$. For each $i\in\bn$, $\cp_i$ is the convex and weakly compact set of probability measures on $(\br,\mathcal{B}(\br))$ characterizing the uncertainty of the distribution of $X_i$. We define a set of joint laws on $(\br^\bn,\mathcal{B}(\br^\bn))$ via the probability kernels as follows:
\begin{equation}
\begin{array}
[c]{l}
\mathcal{P}=\{P\ \text{is\ probability\ measure\ on } (\br^{\bn},\cb(\br^\bn))\ \text{such\ that }
\\P(A^{(n)}\times\br^{\bn-n})=\int_\br\mu_1(dx_1)\int_{\br}\kappa_{2}(x_1,dx_2)\cdots\int_\br I_{A^{(n)}}\kappa_n(x_1,\cdots,x_{n-1},dx_{n})\\
\forall n\geq 1, \ \forall A^{(n)}\in\cb(\br^n),\ \ \kappa_i(x_1,\cdots,x_{i-1},\cdot)\in\cp_i,\ \ 2\leq i\leq n,\ \mu_1\in\cp_1\},
\end{array}
\label{e1}
\end{equation}
where for each $i\in\bn$, $\kappa_i(x_1,\cdots,x_{i-1}, dx_i)$ is the probability kernel satisfying:
\begin{itemize}
\item[(i)] $\forall (x_1,\cdots,x_{i-1})\in\br^{i-1}$, $\kappa_i(x_1,\cdots,x_{i-1},\cdot)$ is a probability measure on $(\br,\mathcal{B}(\br))$.
\item[(ii)] $\forall B\in\mathcal{B}(\br)$, $\kappa_i(\cdot, B)$ is $\mathcal{B}(\br^{i-1})$-measurable.
\end{itemize}
The existence of such $P$ on $(\br^\bn,\cb(\br^\bn))$ is shown by Ionescu-Tulcea theorem.

We note that the uncertainty of the distributions of $X_i$ is independent of random vector $(X_1,\cdots,X_{i-1})$, in the sense of the set $\cp_i$ remaining unchanged for various realization of $(x_1,\cdots,x_{i-1})$ in (\ref{e1}). Such independence is obviously not symmetric in general. Furthermore, if the distributions of each $X_i$ are the same, i.e., $\cp_i=\cp_1$, $\forall i\geq 2$, we say that $\{X_i\}_{i\in\bn}$ is an i.i.d. sequence. In this case, $\cp$ is determined by the set $\cp_1$, characterizing the marginal laws for $X_1$.

\begin{remark}
In particular, if each $\cp_i=\{\mu_i\}$ is a singleton, the above construction degenerates into the classical procedure to construct independent sequence on the product space via the product measure $\mu=\otimes_{i=1}^\infty\mu_i$.
\end{remark}

For each $\mathcal{B}(\br^\bn)$-measurable random variable $X$, we introduce the upper expectation $\pe$ defined by
$$\mathbb{E}^{\cp}[X]:=\sup_{P\in\cp}E_P[X].$$

The canonical filtration $\{\cf_i\}_{i\in\bn}$ is defined as $\cf_i=\sigma(X_k, 1\leq k\leq i)$ with convention $\cf_0=\{\emptyset,\br^\bn\}$.

\subsection{Upper and lower variances with mean-uncertainty}

Let $\cp_0$ be a weakly compact and convex set of probability measures on $(\br,\cb(\br))$ and $X$ be the canonical random variable, i.e., $X(\omega)=\omega, \ \forall \omega\in\br$.
We define
$$\mathbb{E}^{\cp_0}[\vp(X)]=\sup_{P\in\cp_0}E_P[\vp], \ \ \forall \vp\in C(\br).$$

In order to deal with mean-uncertainty case, i.e., $\mathbb{E}^{\cp_0}[X]>-\mathbb{E}^{\cp_0}[-X]$, we need to define the proper upper and lower variances.

We denote $V_P(X)$ the classical variance of $X$ under probability measure $P$, then
$$V_P(X)=E_P[(X-E_P[X])^2]=\min_{\mu\in\br}E_P[(X-\mu)^2].$$
It is natural to introduce the upper and lower variances under sublinear expectation $\mathbb{E}^{\cp_0}$ instead of $E_P$ (see Walley \cite{walley} or Li et al. \cite{lly}).

\begin{definition}\label{uvc}
For the canonical random variable $X$ on $(\br,\cb(\br))$ with $\mathbb{E}^{\cp_0}[|X|^2]<\infty$, define the upper variance of $X$ to be
$$\overline{V}(X):=\min_{\lu\leq\mu\leq\ou}\{\mathbb{E}^{\cp_0}[(X-\mu)^2]\},$$
and the lower variance of $X$ to be
$$\underline{V}(X):=\min_{\lu\leq\mu\leq\ou}\{-\mathbb{{E}}^{\cp_0}[-(X-\mu)^2]\},$$
where $\ou=\mathbb{E}^{\cp_0}[X]$ and $\lu=-\mathbb{E}^{\cp_0}[-X]$.
\end{definition}

Motivated by Walley \cite{walley}, Li et al. \cite{lly} proved that the upper and lower variances can be represented by the envelope of classical variances.

\begin{proposition}\label{prop34}
\begin{equation}
\overline{V}(X)=\max_{P\in\cp_0}V_P(X). \label{ucc1}
\end{equation}
\begin{equation}
\underline{V}(X)=\min_{P\in\cp_0}V_P(X). \label{ucc2}
\end{equation}
\end{proposition}
\begin{remark}
The weak compactness and convexity of $\cp_0$ ensure that the minimax theorem in Sion \cite{sion} holds and (\ref{ucc1}) follows.
For some general $\cp_0$ without such assumption, (\ref{ucc2}) remains true but (\ref{ucc1}) degenerates into the following inequality:
$$\overline{V}(X)\geq\sup_{P\in\cp_0}V_P(X).$$
\end{remark}

The following propositions will show great importance in the proof of the functional CLT.

\begin{proposition}\label{lem35}
For each constant $\sigma^2$ satisfying $\underline{V}(X_1)\leq\sigma^2\leq \overline{V}(X_1)$, there exists $P\in \cp_0$ such that $\sigma^2$ coincides with the variation of $X_1$ under $P$ , i.e.
\begin{equation}\label{conditional2}
E_P[|X_1-E_P[X_1]|^2]=\sigma^2.
\end{equation}
\end{proposition}
\begin{proof}
By Proposition \ref{prop34}, there exist probability measures $\overline{P}$ and $\underline{P}$ in $\cp_0$ such that $V_{\overline{P}}(X_1)=\overline{V}(X_1)$ and $V_{\underline{P}}(X_1)=\underline{V}(X_1)$ respectively. Denote $P$ as the convex combination of $\overline{P}$ and $\underline{P}$, i.e.
\begin{equation}\label{e36}
P(B)=c\underline{P}(B)+(1-c)\overline{P}(B),\ \ \forall B\in\mathcal{B}(\br),
\end{equation}
where $c$ is some constant in $[0,1]$. Then we have $P\in\cp_0$ by convexity, and
$$
\begin{aligned}
E_P[|X_1-E_P[X_1]|^2]&= c \underline{V}(X_1) + (1-c) \overline{V}(X_1) + c(1-c)\lt(E_{\overline{P}}[X_1]-E_{\underline{P}}[X_1]\rt)^2\\
&=: h(c).
\end{aligned}
$$
Noticing that the continuous function $h$ satisfies $h(0)\geq \sigma^2 \geq h(1)$, we can find $c_0 \in [0,1]$ such that $h(c_0) = \sigma^2$. Take $c=c_0$ in (\ref{e36}), then $P$ is the probability measure we need for (\ref{conditional2}) to hold.
\end{proof}

Let $\cp$ be constructed on $(\br^\bn,\cb(\br^\bn))$ by (\ref{e1}) with $\cp_i=\cp_0$ for $i\in\bn$, where $\cp_0$ is weakly compact and convex. Then the  canonical process  $\{X_i\}_{i\in\bn}$ is an i.i.d. sequence under $\cp$. In this and the next section, we further assume that
\begin{equation}\label{c2}
\lim_{\lambda\rightarrow\infty}\mathbb{E}^{\cp_0}[(|X_1|^{2}-\lambda)^+]=0.
\end{equation}


\begin{proposition}\label{lem1}
For each $P\in\cp$, let $\tilde{X}^P_i=X_i-E_P[X_i|\cf_{i-1}]$, then
\begin{equation}\label{conditional1}
\underline{V}(X_1)\leq E_P[|\tilde{X}^P_i|^2|\cf_{i-1}]\leq\overline{V}(X_1),\ \   P-\text{a.s.},  \  \ i\in \bn,
\end{equation}
\end{proposition}
\begin{proof}
If $i=1$, then by Proposition \ref{prop34}, we have
$$\underline{V}(X_1)\leq E_P[(X_1-E_P[X_1])^2]\leq\overline{V}(X_1), \ \ \ \forall P\in\cp.$$

Fix $i\geq 2$. For each $P\in\cp$ with the representation in (\ref{e1}) and $\omega=(x_1,\cdots,x_{i-1},{x}_{i},\cdots)\in\br^\bn$,
\begin{align*}
E_P[|\tilde{X}^P_i|^2|\cf_{i-1}](\omega)=V_{\kappa_i(x_1,\cdots,{x}_{i-1},\ \cdot)}(X_i),\ P-\text{a.s.}.
\end{align*}

Since $\kappa_i(x_1,\cdots,x_{i-1},\cdot)\in\cp_0$,  thanks to Proposition \ref{prop34},  (\ref{conditional1}) holds.

\end{proof}

\section{Functional central limit theorem on the canonical space}

In this section, we establish the functional CLT on the canonical space, which converges to the $G$-expectation characterized by $G$-heat equation. Here we use another representation of $G$-expectation obtained in Denis et al. \cite{DHP11} and then developed by  Dolinsky et al. \cite{DNS}. We adopt the formulations in \cite{DNS}.
\subsection{Representation of $G$-expectation}\label{RG}
Let $\Omega_0=C([0,1])$ be the space of all continuous paths $\omega=(\omega_t)_{0\leq t\leq 1}$ with $\omega_0=0$, endowed with the norm $||\omega||_\infty=\sup_{0\leq t\leq 1}|\omega_t|$. The canonical process $B$ is defined by $B_t(\omega)=\omega_t$ and its quadratic variation process is denoted by $\langle B\rangle$. The canonical filtration is $\cf^B_t=\sigma(B_s,0\leq s\leq t).$

For the fixed non-negative interval $\Theta$, we set
$$\cp_\Theta=\lt\{P: B\ \text{is}\ \text{martingale\ under\ } P \ \text{and}\ \frac{d\langle B\rangle_t}{dt}\in\Theta, P\times dt-\text{a.e.}  \rt\}$$

Let $P_W$ be the Wiener measure on $C([0,1])$ and we define
$$\mathcal{Q}_\Theta=\lt\{P_W\circ\lt(\int f(t,B)dB_t\rt)^{-1}: f\in\mathcal{A}_\Theta \rt\},$$
where $\mathcal{A}_\Theta$ is the collection of all adapted continuous functions on $[0,1]\times\Omega_0$ taking values in $\sqrt{\Theta}$.

It is proved in \cite{DNS} that the upper expectations of $\cp_\Theta$ and $\mathcal{Q}_\Theta$ coincide, i.e.,
$$\mathbb{E}^{\cp_\Theta}[\vp]=\sup_{P\in\cp_\Theta}E_P[\vp]=\sup_{P\in\mathcal{Q}_\Theta}E_P[\vp]=\mathbb{E}^{\mathcal{Q}_\Theta}[\vp], \ \ \ \forall \vp\in C_b(\Omega_0).$$

Indeed, such canonical process $B$ is called a $G$-Brownian motion, which has the independent and stationary increments similar to the classical Brownian motion. The corresponding upper expectation $\mathbb{E}^{\cp_\Theta}$ is called a $G$-expectation (see Peng \cite{P2019}). 

\subsection{Functional central limit theorem with mean-uncertainty}
In order to obtain the functional  CLT, we extend each discrete path $x\in\br^n$ to a continuous path $\hat{x}\in\Omega_0$ using the interpolation operator $\hat{\ }: \br^{n}\rightarrow C([0,1])$ defined as
$$x=(x_1, \cdots, x_n)\mapsto \hat{x}=(\hat{x}_t)_{0\leq t\leq 1},$$
where $\hat{x}_t:=([nt]+1-nt)x_{[nt]}+(nt-[nt])x_{[nt]+1}$ with $x_0=0$.

Equipped with above notations, we have the functional CLT with mean-uncertainty on the canonical space $(\br^\bn,\cb(\br^\bn))$.

\begin{theorem}\label{fclt}
For each $P\in\cp$, let $S_P^{(n)}=(S^P_i)_{1\leq i\leq n}$ with $S^P_i=\frac{1}{\sqrt{n}}\sum_{j=1}^i(X_j-E_P[X_j|\cf_{j-1}])$. Then for each continuous function $\vp:C([0,1])\rightarrow\br$ satisfying $|\vp(\omega)|\leq C(1+||\omega||_\infty)$ for some constant $C>0$, we have
\begin{equation}\label{clt0}
\lim_{n\rightarrow\infty}\sup_{P\in\cp}E_P\lt[\vp\lt(\hat{S_P^{(n)}}\rt)\rt]=\mathbb{E}^{\cp_\Theta}[\vp],
\end{equation}
where $\Theta=[\underline{V}(X_1),\overline{V}(X_1)]$.
\end{theorem}
\begin{proof}
We firstly prove that (\ref{clt0}) holds for $\vp\in C_{b.Lip}(\Omega_0)$, where $C_{b.Lip}(\Omega_0)$ consists of all bounded and Lipschitz continuous functions on $\Omega_0$.

We consider the first inequality
\begin{equation}
\limsup_{n\rightarrow\infty}\sup_{P\in\cp}E_P\lt[\vp\lt(\hat{S_{P}^{(n)}}\rt)\rt]\leq\mathbb{E}^{\cp_{\Theta}}[\vp].\label{iq1}
\end{equation}
Let $\lambda$ be an arbitrary positive constant. Denote $X_n^\lambda:=(X_n\wedge \lambda) \vee (-\lambda)$ and  $S_{P,\lambda}^{(n)}=(S^{P,\lambda}_i)_{1\leq i\leq n}$ with $S^{P,\lambda}_i=\frac{1}{\sqrt{n}}\sum_{j=1}^i(X_j^\lambda-E_P[X_j^\lambda |\cf_{j-1}])$, $n \in \bn$.
For fixed $\vp\in C_{b.Lip}(\Omega_0)$ and $\ve>0$, there exists $P^{(n)}\in\cp$ such that, for each $n\in\bn$,
\begin{equation}\label{eq22}
\sup_{P\in\cp}E_P\lt[\vp\lt(\hat{S_{P,\lambda}^{(n)}}\rt)\rt]\leq E_{P^{(n)}}\lt[\vp\lt(\hat{S_{P^{(n)},\lambda}^{(n)}}\rt)\rt]+\ve.
\end{equation}

Let $Q^{(n)}$ be the law of $\hat{S_{P^{(n)},\lambda}^{(n)}}$ introduced by probability measure $P^{(n)}\in\cp$, i.e.,
$$E_{Q^{(n)}}[\vp(B)]=E_{P^{(n)}}\lt[\vp\lt(\hat{S_{P^{(n)},\lambda}^{(n)}}\rt)\rt], \ \ \forall \vp\in C_b(\Omega_0).$$
The sequence $\{Q^{(n)}\}$ is tight on $C[0,1]$. Let $Q$ be a cluster point of $Q^{(n)}$, then the canonical process $B$ is a $Q$-martingale.

Let
$$M^{(n)}_\lambda(t)=\frac{1}{\sqrt{n}}\sum_{i=0}^{[nt]}\tilde{X}_i^{P^{(n)},\lambda},$$
where
$$\tilde{X}_i^{P^{(n)},\lambda}=X_i^{\lambda}-E_{P^{(n)}}[X_i^{\lambda}|\cf_{i-1}],\ \ X_0=0.$$
The law of $M^{(n)}_\lambda$ on the space $D[0,1]$ of c\`{a}dl\`{a}g paths is denoted as $\tilde{Q}^{(n)}$. We note that
\begin{align*}
&\lt|E_{P_\lambda^{(n)}}\lt[\vp\lt(\hat{S_{P^{(n)},\lambda}^{(n)}}\rt)\rt]-E_{P^{(n)}}[\vp(M^{(n)}_\lambda)]\rt|\\
\leq&\frac{L_\vp}{\sqrt{n}}E_{P^{(n)}}\left[\max_{1\leq i \leq n}|\tX^{P^{(n)},\lambda}_i|\right]\\
\leq&\frac{2L_\vp\lambda}{\sqrt{n}}\rightarrow 0,
\end{align*}
where $L_\vp$ is the Lipschitz constant of $\vp$. Thus $Q$ is also a cluster point of $\tilde{Q}^{(n)}$.

We have
$$\langle M^{(n)}_\lambda\rangle_t=\frac{1}{n}\sum_{i=0}^{[nt]}|\tilde{X}^{P^{(n)},\lambda}_i|^2,$$
and define $d_{\Theta_\lambda}(x)=\inf_{y\in\Theta_\lambda}|y-x|$, where $\Theta_\lambda:=[\underline{V}(X_1^\lambda),\overline{V}(X_1^\lambda)]$.

Since $\underline{V}(X_1^\lambda)\leq E_{\tilde{Q}^{(n)}}\left[|\tilde{X}^{P^{(n)},\lambda}_i|^2|\mathcal{F}_{i-1}\right]\leq\overline{V}(X_1^\lambda)$ by Proposition \ref{lem1}, the similar argument of the law of large numbers  (see Appendix) shows that, for any $0\leq s<t\leq 1$,

$$E_{\tilde{Q}^{(n)}}\lt[d_{\Theta_\lambda}\lt(\frac{\langle M^{(n)}_\lambda\rangle_t-\langle M^{(n)}_\lambda\rangle_s}{t-s}\rt)\rt]\rightarrow 0,$$
which implies that
$$E_{Q}\lt[d_{\Theta_\lambda}\lt(\frac{\langle B\rangle_t-\langle B\rangle_s}{t-s}\rt)\rt]=0.$$
Thus $Q\in \cp_{\Theta_\lambda}$. Hence, by the arbitrariness of $\varepsilon$ in (\ref{eq22}),
\begin{equation}
\limsup_{n\rightarrow\infty}\sup_{P\in\cp}E_P\lt[\vp\lt(\hat{S_{P,\lambda}^{(n)}}\rt)\rt]\leq\mathbb{E}^{\cp_{\Theta_\lambda}}[\vp].
\end{equation}

Let $\lambda\rightarrow \infty$. By (\ref{c2}) we obtain $\Theta_\lambda \rightarrow \Theta$. Thanks to the Rosnethal's inequality for martingales, we also have for each $P\in\cp$,
\begin{equation}
\begin{aligned}
&\lt|E_{P}\lt[\vp\lt(\hat{S_{P}^{(n)}}\rt)\rt]-E_{P}\lt[\vp\lt(\hat{S_{P,\lambda}^{(n)}}\rt)\rt]\rt|\\
\leq&\frac{L_\vp}{\sqrt{n}}E_{P}\lt[\max_{1\leq i \leq n}\lt|\sum_{j=1}^i\lt(\lt(X_j-\lambda\rt)^+-E_{P}\lt[\lt(X_j-\lambda\rt)^+|\cf_{j-1}\rt]\rt)\rt|\rt]\\
\leq& C \lt\{ \e^{\cp_0} \lt[ \lt( | X_1 |^2 - \lambda ^2 \rt) ^{+} \rt]\rt\}^{1/2} \rightarrow 0,
\end{aligned}
\end{equation}
where $C$ is some positive constant. Therefore (\ref{iq1}) holds.

Now we consider the second inequality
\begin{equation}\label{iq2}
\liminf_{n\rightarrow\infty}\sup_{P\in\cp}E_P\lt[\vp\lt(\hat{S_P^{(n)}}\rt)\rt]\geq\mathbb{E}^{\cp_\Theta}[\vp].
\end{equation}
Let $\vp\in C_{b.Lip}(\Omega_0)$ be fixed. For each $\ve>0$, there exists $\bar{f}\in \mathcal{A}_{\Theta}$, such that
$$E_{P_W}\lt[\vp\lt(\int\bar{f}(t,W)dW_t\rt)\rt]\geq \mathbb{E}^{\mathcal{Q}_\Theta}[\vp]-\ve.$$
Without loss of generality, we further assume that $\bar{f}>0$.

Since $\bar{f}(0,\cdot)\in\sqrt{\Theta}$ is a constant, by Proposition \ref{lem35} there exists $P_1 \in \cp_0$ such that
\begin{equation}\label{ce}
{E_{P_1}[|X_1-E_{P_1}[X_1]|^2]}=\bar{f}\lt(0, \cdot\rt)^2.
\end{equation}
Let
$$Y_1=\frac{X_1-E_{P_1}[X_1]}{\sqrt{E_{P_1}[|X_1-E_{P_1}[X_1]|^2]}}.$$

Now we follow a recursive procedure to construct a probability measure $P_*\in \cp$. We define, for $j\geq 2$,
$$P_j(A)=\int_\br P_1(dx_1)\int_\br\kappa_2(x_1,dx_2) \cdots\int_\br I_A\kappa_j(x_1,\cdots,x_{j-1},dx_j),\ \ \forall A \in \mathcal{B}(\br^j)$$
in which $\kappa_j(x_1,\cdots,x_{j-1},\cdot)$ is chosen to be the probability measure $P$ in (\ref{conditional2}) with
$$
\begin{aligned}
&\sigma^2=\bar{f}\lt(\frac{j-1}{n}, \hat{W_{j-1}^{(n)}}\rt)^2\big|_{(X_1,\cdots,X_{j-1})=(x_1,\cdots, x_{j-1})},\\
&W^{(n)}_j(t):=\frac{1}{\sqrt{n}}\sum_{i=0}^{[nt]\wedge j}Y_i,
\end{aligned}
$$
where $Y_0=0$, and for $i\geq 1$,
$$Y_i=\frac{X_i-E_{P_i}[X_i|\cf_{i-1}]}{\sqrt{E_{P_i}[|X_i-E_{P_i}[X_i|\cf_{i-1}]|^2|\cf_{i-1}]}}.$$
Then $P_*$ is formulated by
$$P_*(A\times\br^{\bn-n}):=P_n(A),\ \ \forall A\in\br^n, n\in\bn.$$

Denote
$$W^{(n)}(t)=\frac{1}{\sqrt{n}}\sum_{j=0}^{[nt]}Y_j.$$
We note that $\bar{f}(\frac{j}{n},\hat{W^{(n)}})=\bar{f}(\frac{j}{n},\hat{W_j^{(n)}})$. It is easy to verify that for $1\leq i\leq n$,
\begin{equation*}
{E_{P_*}[|X_i-E_{P_*}[X_i|\cf_{i-1}]|^2|\mathcal{F}_{i-1}]}=\bar{f}\lt((i-1)/n, \hat{W^{(n)}}\rt)^2\in \Theta.
\end{equation*}

Clearly we have $E_{P_*}[Y_j|\cf_{j-1}]=0$ and $E_{P_*}[Y_j^2|\cf_{j-1}]=1$, $j\geq 1$.
By the martingale central limit theorem in Brown \cite{brown}, on the space of c\`{a}dl\`{a}g paths $D([0,1],\br^2)$ equipped with the Skorohod topology,
$$\lt(W^{(n)},\hat{W^{(n)}}\rt)\Rightarrow(W,W),$$
where $W$ is the standard Brownian motion.

We note that
\begin{align*}
S^{P_*}_i&=\frac{1}{\sqrt{n}}\sum_{j=1}^i(X_j-E_{P_*}[X_j|\cf_{j-1}])\\
&=\sum_{j=1}^i\bar{f}\lt(\frac{j-1}{n}, \hat{W^{(n)}}\rt)\lt(W^{(n)}\lt(\frac{j+1}{n}\rt)-W^{(n)}\lt(\frac{j}{n}\rt)\rt).
\end{align*}
Then it follows from Dolinsky et al. \cite{DNS} and the stability of stochastic integrals (see Kurtz and Protter \cite{KP}) that, on $D(0,1)$,
$$\hat{S_{P_*}^{(n)}}\Rightarrow \int\bar{f}(t,W)dW_t.$$
 Finally, we obtain
\begin{align*}
\liminf_{n\rightarrow\infty}\sup_{P\in\cp}E_P\lt[\vp\lt(\hat{S_P^{(n)}}\rt)\rt]&\geq\liminf_{n\rightarrow\infty}E_{P_*}\lt[\vp\lt(\hat{S_{P_*}^{(n)}}\rt)\rt]\\
&= E_{P_W}\lt[\vp\lt(\int_0^1\bar{f}(t,W)dW_t\rt)\rt]\geq \mathbb{E}^{\mathcal{Q}_\Theta}[\vp]-\ve.
\end{align*}
Thus (\ref{clt0}) holds for $\vp\in C_{b.Lip}(\Omega_0)$.

For fixed $n\in\bn$ and $\lambda>0$, by Doob's martingale inequality, we have
$$E_{P}\lt[\sup_{0\leq t\leq 1}|M^{(n)}_\lambda(t)|^2\rt]\leq 4\overline{V}(X_1^\lambda), \ \ \ \forall P\in\cp,$$
which implies that
$$\sup_{P\in\cp}E_P\lt[\sup_{0\leq t\leq 1}\lt|\hat{S_P^{(n)}}(t)\rt|^2\rt]\leq 4\overline{V}(X_1^\lambda)+2\lambda+ C^{\prime} \e^{\cp_0} \lt[ \lt( | X_1 |^2 - \lambda ^2 \rt) ^{+} \rt]<\infty.$$
By the similar arguments as Lemma 2.4.12 in Peng \cite{P2019}, (\ref{clt0}) also holds for continuous function $\vp$ with linear growth condition.
\end{proof}


\section{Functional central limit theorem on sublinear expectation space}

In this section, we extend  functional CLT to the general sublinear expectation space $(\Omega,\mathcal{H},\be)$ introduced by Peng \cite{P2019}.
\subsection{Basic notions of sublinear expectation theory}
Let $\Omega$ be a given set and let $\mathcal{H}$ be a linear space of real
functions defined on $\Omega$ such that if
$X_{1},\ldots,X_{n}\in \mathcal{H}$, then $\varphi(X_{1},\cdots,X_{n}%
)\in \mathcal{H}$ for each $\varphi \in C_{Lip}(\mathbb{R}^{n})$, where
$C_{Lip}(\mathbb{R}^{n})$ denotes the space of all Lipschitz
functions on $\mathbb{R}^{n}$. $\mathcal{H}$ will serve as the space of
random variables.
$X=(X_{1},\ldots,X_{n})$, $X_{i}\in \mathcal{H}$, is called
an $n$-dimensional random vector.

\begin{definition}\label{sub}
A sublinear expectation on $\mathcal{H}$ is a functional $\hat{\mathbb{{E}}}:\mathcal{H}\rightarrow \mathbb{R}$ satisfying the following
properties: for all $X,Y\in \mathcal{H}$,
\begin{description}
\item[(a)] Monotonicity: $\mathbb{\hat{E}}[X]\geq \mathbb{\hat{E}}[Y]$ if
$X\geq Y$.

\item[(b)] Constant preserving: $\mathbb{\hat{E}}[c]=c$ for $c\in \mathbb{R}$.

\item[(c)] Sub-additivity: $\mathbb{\hat{E}}[X+Y]\leq \mathbb{\hat{E}%
}[X]+\mathbb{\hat{E}}[Y]$.

\item[(d)] Positive homogeneity: $\mathbb{\hat{E}}[\lambda X]=\lambda
\mathbb{\hat{E}}[X]$ for $\lambda \geq0$.
\end{description}

The triple $(\Omega,\mathcal{H},\mathbb{\hat{E}})$ is called a sublinear
expectation space.
\end{definition}


Let $X=(X_1,\cdots,X_n)$ be a given $n$-dimensional random vector on a sublinear expectation space $(\Omega,\mathcal{H},\mathbb{\hat{E}})$. We define a functional on $C_{b.Lip}(\br^n)$ by
$$\mathbb{\hat{F}}_X[\vp]:=\mathbb{\hat{E}}[\vp(X)], \ \ \forall \vp\in C_{b.Lip}(\br^n).$$
The triple $(\br^n, C_{b.Lip}(\br^n),\mathbb{\hat{F}}_X[\cdot])$ forms a sublinear expectation space, and $\mathbb{\hat{F}}_X$ is called the sublinear distribution of $X$.

\begin{definition}\label{id}
Let $X$ and $Y$ be two random variables on $(\Omega,\mathcal{H},\mathbb{\hat
{E}})$. $X$ and $Y$ are called identically distributed, denoted by
$X\overset{d}{=}Y$, if for each $\varphi \in C_{b.Lip}(\mathbb{R})$,
\[
\mathbb{\hat{E}}[\varphi(X)]=\mathbb{\hat{E}}[\varphi(Y)].
\]

\end{definition}

\begin{definition}
\label{new-de1}Let $\left \{  X_{n}\right \}  _{n\in\bn}$ be a sequence of
random variables on $(\Omega,\mathcal{H},\mathbb{\hat{E}})$. $X_{n}$ is said
to be independent of $\left(  X_{1},\ldots,X_{n-1}\right)  $ under
$\mathbb{\hat{E}}$, if for each $\vp\in C_{b.Lip}(\br^{n})$,
$$
\mathbb{\hat{E}}\left[\varphi\left(X_{1}, \cdots, X_{n}\right)\right]=\mathbb{\hat{E}}\left[\left.\mathbb{\hat{E}}\left[\varphi\left(x_{1}, \cdots, x_{n-1},X_n\right)\right]\right|_{\left(x_{1}, \cdots, x_{n-1}\right)=\left(X_{1}, \cdots, X_{n-1}\right)}\right].
$$
The sequence of random variables $\left \{  X_{n}\right \}  _{n\in\bn}$ is
said to be independent, if $X_{n+1}$ is independent of $\left(  X_{1}%
,\ldots,X_{n}\right)  $ for each $n\geq1$.
\end{definition}

\begin{remark}
Compared with the classical definition of independence, the one we provide above is more general, since the regularity of $\be$ (see Definition \ref{reg}) is not required. Example \ref{ex1} illustrates that a sequence of random variables might be able to be viewed i.i.d. from the perspective of sublinear expectation, even when that is impossible under classical probability measure.
\end{remark}

\begin{definition}
A  random variable $\xi$ on a sublinear
expectation space $(\Omega,\mathcal{H},\be)$ with $\lim_{\lambda\to\infty}\be[(|\xi|^2-\lambda)^+]=0$ is called
{$G$-normally distributed}, if
\begin{equation}\label{sp}
a\xi+b\bar{\xi}\overset{d}{=}\sqrt{a^{2}+b^{2}}\xi,\ \ \ \ {\forall }a,b>0,\
\end{equation}
where $\bar{\xi}$ is independent and identically distributed with $\xi$.
\end{definition}

The $G$-normal distributed random variable $\xi$ is determined by two parameters $\os^2=\be[\xi^2]$ and $\ls^2=-\be[-\xi^2]$, denoted by $\xi\sim\mathcal{N}(0,[\ls^2,\os^2]).$

\begin{proposition}
Let $\xi$ be $G$-normally distributed on the sublinear expectation space $(\Omega,\mathcal{H},\be)$ with $\overline{\sigma}^2=\be[\xi^2]$ and $\underline{\sigma}^2=-\be[-\xi^2]$. For each $\vp\in C_{Lip}(\mathbb{R})$, we define
$$u(t,x)=\be[\vp(x+\sqrt{1-t}\xi)], (t,x)\in[0,1]\times\br.$$
Then $u(t,x)$ is the unique viscosity solution of following parabolic PDE:
\begin{equation}\label{g-heat}
\partial_tu+G(\partial_{xx}^2u)=0,
\end{equation}
with Cauchy condition $u|_{t=1}=\vp$, where $G(\alpha)=\frac{1}{2}(\os^2\alpha^+-\ls^2\alpha^-), \ \ \alpha\in\br.$
\end{proposition}
We usually denote $\be[\vp(\xi)]$ by $\eg[\vp(\xi)]$ and call it $G$-expectation. Correspondingly, the ``dynamic'' $G$-normal distribution, called $G$-Brownian motion, is well defined. See more details about these notions in Peng \cite{P2019}. Indeed, characterizations of $G$-expectation and  $G$-Brownian motion here and those in Section \ref{RG} coincide in canonical space.

For random variable $X$ on sublinear expectation $(\Omega,\mathcal{H},\be)$, the corresponding upper and lower variances of $X$ can be defined similarly as in Definition \ref{uvc}, specifically,
$$\overline{V}(X)=\min_{\lu\leq\mu\leq\ou}\be[(X-\mu)^2] \ \ \text{and} \ \  \underline{V}(X)=\min_{\lu\leq\mu\leq\ou}-\be[-(X-\mu)^2].$$

\subsection{Functional central limit theorem under regular sublinear expectation}
We firstly assume that sublinear expectation $\be$ is regular and present the corresponding functional CLT.

\begin{definition}\label{reg}
The sublinear expectation $\be$ is said to be regular if for each sequence $\{X_i\}_{i\in\bn}\subset\mathcal{H}$ with $X_i\downarrow 0$, we have $\be[X_i]\downarrow 0$.
\end{definition}

Assume $\be$ to be regular. We consider a family of probability measures $\cp_*$ defined on the measurable space $(\Omega,\sigma(\mathcal{H}))$:
\begin{equation}\label{eqp}
\cp_*=\{P \ \text{is\ probability\ measure\ on\ }(\Omega,\sigma(\mathcal{H})): E_P[X]\leq \be[X], \ \ \forall X\in\mathcal{H}\},
\end{equation}
where $\sigma(\mathcal{H})$ is the smallest $\sigma$-algebra generated by $\mathcal{H}$.

By the Robust Daniell-Stone theorem in Peng \cite{P2019}, $\cp_*$ is non-empty. Define
$$\be[X]=\sup_{P\in\cp_*}E_P[X], \ \ \forall X\in\mathcal{H}.$$
Then we enlarge the domain of $\be$ from $\mathcal{H}$ to $\mathcal{L}(\sigma(\mathcal{H}))$ as follows:
$$\be[X]=\sup_{P\in\cp_*}E_P[X], \ \forall X\in\mathcal{L}(\sigma(\mathcal{H})),$$
where $\mathcal{L}(\sigma(\mathcal{H}))$ is the space of all $\sigma(\mathcal{H})$-measurable functions.

Proposition \ref{lem1} can be naturally generalized on the regular sublinear expectation space $(\Omega,\mathcal{H},\be)$ (see Guo et al. \cite{gll}).
\begin{proposition}\label{lem2}
Let $\{X_i\}_{i\in\bn}$ be an i.i.d. sequence on sublinear expectation space $(\Omega,\mathcal{H},\be)$. We further assume that
\begin{equation*}
\lim_{\lambda\rightarrow\infty}\be[(|X_1|^{2}-\lambda)^+]=0.
\end{equation*}
For fixed $n\in\bn$, and for each $P\in\cp_*$, let $\tX^P_i=X_i-E_P[X_i|\cf_{i-1}], \ 1\leq i\leq n$. Then we have
\begin{equation*}
\underline{V}(X_1)\leq E_P[|\tX^P_i|^2|\cf_{i-1}]\leq\overline{V}(X_1),\ \   P-\text{a.s.},  \  \ 1\leq i\leq n.
\end{equation*}
\end{proposition}

For each $P\in\cp_*$, let $P\circ \boldsymbol{X}^{-1}$ be the measure on the canonical space $(\br^\bn,\mathcal{B}(\br^\bn))$ introduced by a random vector  $\boldsymbol{X}=(X_1,\cdots,X_n,\cdots)$, where $\{X_i\}_{i\in\bn}$ is an i.i.d. sequence on the sublinear expectation space $(\Omega,\mathcal{H},\be)$. 
Define
\begin{equation}\label{PX}
\cp_*\circ{\boldsymbol{X}^{-1}}:=\{P\circ \boldsymbol{X}^{-1}:\ \ \ P\in\cp_*\}.
\end{equation}

Let $\cp$ be the set defined via the probability kernels as in (\ref{e1}) with $\cp_i=\cp_1, \forall i\geq 2$, where $\cp_1$ represents the distribution of $X_1$ under $\be[\cdot]$ in the following sense:
$$\cp_1=\{\mu: E_\mu[\vp]\leq\be[\vp(X_1)], \ \ \forall \vp\in C_{Lip}(\br)\}.$$

We obtain a new representation theorem.

\begin{proposition}\label{l-m}
Assume that the i.i.d. sequence $\{X_i\}_{i\in\bn}$ on the sublinear expectation space $(\Omega,\mathcal{H},\be)$ satisfies
\begin{equation}\label{ui-4}
\lim_{\lambda\rightarrow\infty}\be[(|X_1|-\lambda)^+]=0.
\end{equation}
Denote $\mathcal{Q}_{\max}$ as the set of probability measures on the canonical space $(\br^\bn,\mathcal{B}(\br^\bn))$ defined by
$$\mathcal{Q}_{\max}=\{\mu: E_\mu[\vp]\leq\be[\vp(X_1,\cdots,X_n)], \forall n\in\bn, \forall \vp\in C_{Lip}(\br^n)\}.$$
Then we have
\begin{equation}\label{pq}
\cp_*\circ \boldsymbol{X}^{-1}=\cp=\mathcal{Q}_{\max}.
\end{equation}
\end{proposition}
\begin{proof}
Let $\{Y_i\}_{i\in\bn}$ be the canonical random variables on $(\br^\bn,\cb(\br^\bn))$.

According to Theorem 2.9 in Guo et al. \cite{gll3}, we obtain
$$\mathcal{Q}_{\max}=\{\mu: E_{\mu}[\vp(Y_n)|Y_1,\cdots,Y_{n-1}]\leq\mathbb{E}^{\cp}[\vp(Y_1)], \ \forall n\in\bn, \forall \vp\in C_{Lip}(\br^n)\},$$
with convention $E_\mu[\vp(Y_1)|Y_0]=E_\mu[\vp(Y_1)]$ if $n=1$.

For each measure $Q \in\mathcal{Q}_{\max}$, we have its desintegration as, $\forall n\in\bn$, $ \forall A\in\cb(\br^n)$,
$$Q(A\times\br^{\bn-n})=\int_\br\mu_1(dx_1)\int_{\br}\kappa_{2}(x_1,dx_2)\cdots\int_\br I_{A}\kappa_n(x_1,\cdots,x_{n-1},dx_{n}),$$
where $\kappa_i, i\geq 2$ are probability kernels and $\mu_1$ is the probability measure on $(\br,\mathcal{\br})$.

Since $E_{\mu_1}[\vp]\leq\mathbb{E}^{\cp}[\vp(Y_1)]=\be[\vp(X_1)]$,  $\forall \vp\in C_{Lip}(\br)$, we obtain $\mu_1\in\cp_1$. For each $n\geq 2$, we notice
$$\begin{aligned}
E_Q[\vp(Y_n)|Y_1,\cdots,Y_{n-1}]=& \int_\br\vp(x)\kappa_n(x_1,\cdots,x_{n-1},dx)|_{x_1=Y_1,\cdots,x_{n-1}=Y_{n-1}}\\
\leq & \mathbb{E}^{\cp}[\vp(Y_n)]=\be[\vp(X_1)], \ \ \ \ \forall \vp\in C_{Lip}(\br),
\end{aligned}$$
hence, $\kappa_n\in\cp_1$, which implies $\cp\supset\mathcal{Q}_{\max}$, and the opposite is trival. Therefore, we have $\cp=\mathcal{Q}_{\max}$.

Obviously, all of $\cp_*\circ \boldsymbol{X}^{-1}$, $\cp$ and $\mathcal{Q}_{\max}$ are convex.

Noting that
$$\begin{aligned}
\sup_{\mu \in \cp_*\circ \boldsymbol{X}^{-1}} \mu \lt(\bigcap_{i=1}^{\infty}\{|Y_i|>\lambda\}\rt)&\leq \sup_{\mu \in \cp_*\circ \boldsymbol{X}^{-1}}\sup_{i\in\bn} \ \mu \lt(|Y_i|>\lambda\rt) \\
&= \sup_{P \in \cp_*}P \lt(|X_1|>\lambda\rt)\leq
2 \be\left[\left(|X_1|-\frac{\lambda}{2}\right)^+\right]\to 0,
\end{aligned}$$
as $\lambda\to \infty$, the relatively weak compactness of $\cp_*\circ \boldsymbol{X}^{-1}$ follows.
By the same argument of the proof of Lemma 4.2(b) in Zhang \cite{Zhang2021}, we yield that $\cp_*\circ \boldsymbol{X}^{-1}$ is also close, therefore it is weakly compact.

Following the similar argument, we conclude that $\mathcal{Q}_{\max}$ is weakly compact as well.

It can be easily seen that
$$\mathbb{E}^{\cp_*\circ \boldsymbol{X}^{-1}}[\vp]=\mathbb{E}^\cp[\vp]=\mathbb{E}^{\mathcal{Q}_{\max}}[\vp], \ \ \forall n\in\bn, \vp\in C_{Lip}(\br^n).$$
So $\cp_*\circ \boldsymbol{X}^{-1}=\cp=\mathcal{Q}_{\max}$ by Corollary 2.8 in \cite{LL}.

\end{proof}

\begin{remark}\label{r48} With (\ref{pq}), we have
\begin{equation}\label{l-id}
\be[f(X_1,\cdots,X_n)]=\mathbb{E}^{\cp}[f(Y_1,\cdots,Y_n)]
\end{equation}
holds for each $n\in\bn$ and bounded $\cb(\br^n)$-measurable function $f$.

Theorem 3.1 in Guo et al. \cite{gll3} shows that for general sublinear expectation $\be$ and i.i.d sequence $\{X_i\}_{i\in\bn}$,  (\ref{l-id}) is valid for $f\in C_{Lip}(\br^n)$. Proposition \ref{l-m} is actually an extension of such result.

\end{remark}

Now the functional CLT with mean-uncertainty on the canonical space can be easily extended.

\begin{theorem}
Let $\{X_i\}_{i\in\bn}$ be an i.i.d. sequence on sublinear expectation space $(\Omega,\mathcal{H},\be)$ with
$$\lim_{\lambda\rightarrow\infty}\be[(|X_1|^2-\lambda)^+]=0.$$
We further assume that $\be$ is regular and associated with $\cp_*$ defined by (\ref{eqp}).
For each $P\in\cp_*$, let $\tilde{X}^P_i=X_i-E_P[X_i|\cf_{i-1}]$ and $S^{(n)}_P=(S^P_i)_{1\leq i\leq n}$ with $S^P_i=\frac{1}{\sqrt{n}}\sum_{j=1}^i\tilde{X}^P_j$. Then for each continuous function $\vp:C([0,1])\rightarrow\br$ satisfying $|\vp(\omega)|\leq C(1+||\omega||_\infty)$ for some constants $C>0$,
\begin{equation}\label{fclt-r}
\lim_{n\rightarrow\infty}\sup_{P\in\cp_*}E_P\lt[\vp\lt(\hat{S_P^{(n)}}\rt)\rt]= \mathbb{E}^{\cp_\Theta}[\vp],
\end{equation}
where $\Theta=[\underline{V}(X_1),\overline{V}(X_1)]$.
\end{theorem}

\begin{proof}

Let $\{Y_i\}_{i\in\bn}$ be the canonical random variables on $(\br^\bn,\cb(\br^\bn))$ and $\cp_*\circ \boldsymbol{X}^{-1}$,$\cp$ be constructed as in Proposition \ref{l-m}.

For each $P\in\cp$, let $\tilde{Y}^P_i=Y_i-E_P[Y_i|\mathcal{G}_{i-1}]$, where $\mathcal{G}_{i-1}=\sigma(Y_i, 1\leq j\leq i)$, and $T_P^{(n)}=(T^P_i)_{1\leq i\leq n}$, where $T^P_i=\frac{1}{\sqrt{n}}\sum_{j=1}^i \tilde{Y}^P_j$.

Notice that if $\mu$ is the probability measure on $(\br^\bn,\cb(\br^\bn))$ induced by $\{X_i\}_{i\in\bn}$ under $P$, then $E_P[X_i|\mathcal{F}_{i-1}]\overset{\text{d}}{=} E_\mu[Y_i|\mathcal{G}_{i-1}]$.
Therefore, for $\forall \vp\in C_{b.Lip}(\Omega_0)$,
$$\begin{aligned}
\sup_{P\in\cp_*}E_P\lt[\vp\lt(\hat{S_P^{(n)}}\rt)\rt]&=\sup_{\mu\in\cp_*\circ \boldsymbol{X}^{-1}}E_{\mu}\lt[\vp\lt(\hat{T_P^{(n)}}\rt)\rt]\\
&= \sup_{P\in\cp}E_P\lt[\vp\lt(\hat{T_P^{(n)}}\rt)\rt]\rightarrow\mathbb{E}^{\cp_\Theta}[\vp],\ \  n\rightarrow\infty,
\end{aligned}$$
where the convergence is guaranteed by Theorem \ref{fclt}. The proof now is completed.

\end{proof}

\subsection{Functional central limit theorem without assumption of regularity}
Now we remove the assumption of regularity of $\be$. For fixed sequence of random variables $\{X_i\}_{1\leq i\leq n}$, let $\cp_*^{(n)}$ denote the class of all probability measures on $(\Omega,\cf_n)$ that dominated by $\be$ in the following sense:
\begin{align*}
\cp_*^{(n)}=\{P: &E_P[\vp(X_1,\cdots,X_n)]\leq\be[\vp(X_1,\cdots,X_n)], \ \forall \vp\in C_{b.Lip}(\br^n)\}.
\end{align*}
where $\cf_n=\sigma(X_1,\cdots,X_n)$ is the natural filtration with convention $\cf_0=\{\emptyset,\Omega\}$.

Denote $\mathcal{L}(\cf_n)$ the space of all $\cf_n$-measurable functions. Then $\be$ is regular on $\mathcal{H}\bigcap\mathcal{L}(\cf_n)$ (see Theorem 10 in Hu and Li \cite{HL}) and the domain of $\be$ can be enlarged naturally from $\mathcal{H}\bigcap\mathcal{L}(\cf_n)$ to $\mathcal{L}(\cf_n)$, specifically,
$$\be[X]=\sup_{P\in\cp_*^{(n)}}E_P[X], \ \forall X\in\mathcal{L}(\cf_n).$$

We immediately have the functional CLT without the regularity of $\be$ by the similar argument as in the previous subsection.
\begin{theorem}\label{fcltsub}
Let $\{X_i\}_{i\in\bn}$ be an i.i.d. sequence on sublinear expectation space $(\Omega,\mathcal{H},\be)$ with
$$\lim_{\lambda\rightarrow\infty}\be[(|X_1|^2-\lambda)^+]=0,$$
and $\cp_*^{(n)}$ be the probability set dominated by $\be$ on $\mathcal{H}\cap\mathcal{L}(\cf_n)$. For each $P\in\cp_*^{(n)}$, let $\tilde{X}^P_i=X_i-E_P[X_i|\cf_{i-1}]$ and $S^{(n)}_P=(S^P_i)_{1\leq i\leq n}$ with $S^P_i=\frac{1}{\sqrt{n}}\sum_{j=1}^i\tilde{X}^P_j$. Then for each continuous function $\vp:C([0,1])\rightarrow\br$ satisfying $|\vp(\omega)|\leq C(1+||\omega||_\infty)$ for some constants $C>0$,
$$\lim_{n\rightarrow\infty}\sup_{P\in\cp_*^{(n)}}E_P\lt[\vp\lt(\hat{S_P^{(n)}}\rt)\rt] = \mathbb{E}^{\cp_\Theta}[\vp],$$
where $\Theta=[\underline{V}(X_1),\overline{V}(X_1)]$.
\end{theorem}

Now we generalize CLT with zero-mean in Peng \cite{P2019} to the case of mean-uncertainty.
\begin{corollary}\label{cltv}
With the same assumptions in Theorem \ref{fcltsub}, for each $\vp\in C(\br)$ with linear growth we have
\begin{equation}
\label{th1}\lim_{n\rightarrow\infty}\sup_{P\in\cp_*^{(n)}}E_P\lt[\vp\lt(\frac{\sum_{i=1}^n(X_i-E_P[X_i|\cf_{i-1}])}{\sqrt{n}}\rt)\rt]= \eg[\vp(\xi)],
\end{equation}
where  $\xi\sim\mathcal{N}(0,[\underline{V}(X_1),\overline{V}(X_1)])$.
\end{corollary}

{{The conditional expectation is known to be the least-squares-best predictor of $X_i$ based on the realization of the historical data $\{X_j\}_{1\leq j\leq i-1}$, therefore can serve as the modification parameter. Formula (\ref{th1}) implies that the properly normalized sum for i.i.d. random variables tends toward the $G$-normal distribution. Corollary \ref{cltv} indicates the broad applicability of $G$-normal distribution. In particular, for i.i.d. random variables with mean-certainty, the corresponding least-square-best predictor are constants equal to their certain mean. That is the exact CLT with certain means in Peng \cite{P2019}. }}

We also have the following CLT in the sense of capacity.
\begin{corollary}\label{coro}
Under the same assumptions of Theorem \ref{fcltsub}, we have
\begin{equation}\label{cap}
\lim_{n\rightarrow\infty}\sup_{P\in\cp_*^{(n)}}P\lt(a\leq\frac{\sum_{i=1}^n(X_i-E_P[X_i|\cf_{i-1}])}{\sqrt{n}}\leq b\rt)=V_\Theta([a,b]),
\end{equation}
where $V_\Theta(A):=\sup_{P\in\cp_\Theta}P(A), \ \forall A\in\cb(\br)$ with $\Theta=[\underline{V}(X_1),\overline{V}(X_1)]$ and $-\infty\leq a<b\leq\infty$.

In particular,
\begin{equation}\label{G}
\lim_{n\rightarrow\infty}\sup_{P\in\cp_*^{(n)}}P\lt(\frac{\sum_{i=1}^n(X_i-E_P[X_i|\cf_{i-1}])}{\sqrt{n}}\leq x\rt)=\lt\{
\begin{aligned}
&\frac{2 \overline{\sigma}}{\overline{\sigma}+\underline{\sigma}} \Phi\left(\frac{x}{\overline{\sigma}}\right),\ &x\leq 0,\\
&1-\frac{2 \underline{\sigma}}{\overline{\sigma}+\underline{\sigma}} \Phi\left(-\frac{x}{\underline{\sigma}}\right), \ &x>0,
\end{aligned}
\rt.
\end{equation}
where $\Phi$ denotes the distribution function of the standard normal and $\os=\sqrt{\overline{V}(X_1)}, \ls=\sqrt{\underline{V}(X_1)}$.
\end{corollary}

\begin{proof}
For any $0<\ve<\frac{b-a}{2}$, there exist $f^\ve,g^\ve\in C_{b.Lip}(\br)$ such that
$${I}_{[a+\ve,b-\ve]}(x)\leq g^\ve(x)\leq {I}_{[a,b]}(x)\leq f^\ve(x)\leq {I}_{[a-\ve,b+\ve]}(x).$$
Then we have
\begin{align*}\eg[g^\ve(\xi)]&=\liminf_{n\rightarrow\infty}\sup_{P\in\cp_*^{(n)}}E_P\lt[g^\ve\lt(\frac{\sum_{i=1}^nX_i-E_P[X_i|\cf_{i-1}]}{\sqrt{n}}\rt)\rt]\\
&\leq\liminf_{n\rightarrow\infty}\sup_{P\in\cp_*^{(n)}}P\lt(a\leq\frac{\sum_{i=1}^nX_i-E_P[X_i|\cf_{i-1}]}{\sqrt{n}}\leq b\rt)\\
&\leq \limsup_{n\rightarrow\infty}\sup_{P\in\cp_*^{(n)}}P\lt(a\leq\frac{\sum_{i=1}^nX_i-E_P[X_i|\cf_{i-1}]}{\sqrt{n}}\leq b\rt)\\
&\leq\limsup_{n\rightarrow\infty}\sup_{P\in\cp_*^{(n)}}E_P\lt[f^\ve\lt(\frac{\sum_{i=1}^nX_i-E_P[X_i|\cf_{i-1}]}{\sqrt{n}}\rt)\rt]=\eg[f^\ve(\xi)]
\end{align*}
We note that
$$V_\Theta([a+\ve,b-\ve])\leq\eg[g^\ve(\xi)]\leq\eg[f^\ve(\xi)]\leq V_\Theta([a-\ve,b+\ve]),$$
and for each $x\in\br$, by Corollary 3.8 in Hu et al. \cite{HWZ}, $V_\Theta([x-\ve,x+\ve])\rightarrow 0$ as $\ve\rightarrow 0$, which implies (\ref{cap}).

(\ref{G}) is the direct deduction of (\ref{cap}) combining with the calculation of $V_{\Theta}$ in Peng et al. \cite{Yang}.
\end{proof}


\section{Applications and examples}

In this section, we will present some applications of our functional CLT  and a counterexample to show that CLT does not hold with finite second moment condition.

\subsection{Applications: Two-armed bandit problem}

Motivated by some practical problem concerning sequential choice, such as clinical trial, the bandit problem attracted a lot of attention of researchers (see Chen et al. \cite{CEZ}).  As an application of our results, we consider a CLT for two-armed bandit problem including the case of mean-uncertainty, which generalizes the existing results of the case of mean-certainty in \cite{CEZ}.

Let $(\Omega_0, \mathcal{F}, P)$ be a probability space. The random payoffs from arms L and R are represented by two bounded random variables $W^L$ and $W^R$ respectively. Let $\{W_i^L\}_{i=1}^{\infty}$ and $\{W_i^R\}_{i=1}^{\infty}$ be the i.i.d copies of $W^L$ and $W^R$ lie in the product space $\Omega=\prod_{i=1}^{\infty}\Omega_i$, where $\Omega_i=\Omega_0$, $W_i^L$ and $W_i^R$ denotes the random payoffs from arms L and R at round $i$. A sampling strategy is defined as a infinite sequence $s:=\{s_i\}_{i=1}^{\infty}$, and the reward at round $i$ under strategy $s$ is defined as
$$Z_i^s:=s_i W_i^L + (1-s_i)W_i^R, $$
where $s_i$, with value $0$ or $1$, denotes the choice of action at round $i$ depending on the history of past actions and rewards $(s_1,Z_1^s,\cdots,s_{i-1}, Z_{i-1}^s)$. The corresponding set of strategies is denoted as $\cs$.

Let $\{X_i\}_{i=1}^{\infty}$ denotes the canonical sequence on the canonical space $(\br^\bn,\cb(\br^\bn))$. Set
$$\cp_0:=\lt\{\mu:E_\mu[\vp(X_1)]\leq \be[\vp(X_1)]:=E_P[\vp(W^L)]\vee E_P[\vp(W^R)],\ \forall \vp \in C_{{Lip}}(\br)\rt\}.$$
Construct $\cp$ the set of probability measure on $\br^\bn$ following the procedure in (\ref{e1}) with $\cp_i=\cp_0$, $i\geq 1$. Then $\{X_i\}_{i=1}^{\infty}$ becomes an i.i.d sequence on the sublinear expectation space $(\br^\bn,\cb(\br^\bn),\EE)$. 
By a simple calculation, we obtain
$$\ou:=\EE[X_1]=\mu_L\vee \mu_R,\ \ \ \lu: =-\EE[-X_1]=\mu_L\wedge \mu_R,$$
$$\os^2:=\overline{V}(X_1)=\lt\{
\begin{aligned}
&\frac{1}{4}\lt[(\mu_L-\mu_R)^2 +\lt(\frac{\sigma^2_L -\sigma^2_R}{\mu_L -\mu_R}\rt)^2\rt]+\frac{\sigma^2_L+\sigma^2_R}{2},\ \text{if } \mu_L \neq \mu_R,\\
&\sigma^2_L \vee \sigma^2_R,\ \text{if } \mu_L = \mu_R,
\end{aligned}\rt.$$
$$\ls^2:=\underline{V}(X_1)=\sigma^2_L \wedge \sigma^2_R,$$
with
$$\mu_L:=E_P [W^L],\ \ \ \mu_R:=E_P [W^R],$$
$$\sigma^2_L:=\operatorname{Var}(W^L),\ \ \ \sigma^2_R:=\operatorname{Var}(W^R).$$

For each strategy $s\in\cs$, the sequence $\{Z_i^s\}_{i=1}^{\infty}$ induces a probability measure $P^s$ on $(\br^\bn,\cb(\br^\bn))$ under $P$. Therefore
$$\begin{aligned}
E_{P^s}[\vp(X_1,X_2)]=&E_P[\vp(Z_1^s, Z_2^s)]=E_P\lt[E_P\lt[\vp(Z_1^s, Z_2^s)|Z_1^s\rt]\rt]\\
=&E_P\lt[E_P\lt[\vp(x, y W_2^L + (1-y)W_2^R)\rt]|_{x=Z_1^s,y=s_2}\rt]\\
\leq & E_P\lt[\mathbb{E}^{\cp_0}\lt[\vp(x, X_2)\rt]|_{x=s_1 W_1^L + (1-s_1)W_1^R}\rt]\\
\leq &\mathbb{E}^{\cp_0}\lt[\mathbb{E}^{\cp_0}\lt[\vp(x, X_2)\rt]|_{x=X_1}\rt]=\EE[\vp(X_1,X_2)],
\end{aligned}$$
since $s_1,s_2$ only take value $0$ or $1$, $s_2$ depends on $(s_1,Z_1^s)$ and $W_2^L,W_2^R$ are independent of $Z_1^s$. By a recursive argument, we prove that
$$E_{P^s}[\vp(X_1,\cdots, X_i)]\leq \EE [\vp(X_1,\cdots, X_i)],\ \ \forall i\geq 1.$$
Noticing the convexity and weak compactness of $\cp$, we have $P^s \in\cp$. Hence for each $i\geq 1$ and $\cb(\br^i)$-measurable function $f$, we obtain
$$\sup_{s\in\cs}E_P\lt[f\lt(Z_1^s,\cdots,Z_i^s\rt)\rt]\leq \EE \lt[f\lt(X_1,\cdots,X_i\rt)\rt].$$

On the other hand, for fixed $i\geq 1$ and $\cb(\br^i)$-measurable function $f$, 
we construct a $i$-step strategy $s^*\in \cs$ as follows:
\begin{enumerate}
  \item[$\bullet$] At round $1$, $s^*_1=1$ if $$E_P[h_1(W^L) - h_1(W^R)]\geq 0$$ and $s^*_1=0$ if otherwise, where $$h_1(x):=
      \EE\lt[f(x, X_2,\cdots,X_i)\rt]$$
  \item [$\bullet$] At round $2$, $s^*_2=1$ if
  $$E_P\lt[h_2(x_1,W^L)- h_2(x_1,W^R)\rt]\Big|_{x_1=Z_1^{s^*}}\geq 0$$
      and $s^*_2=0$ if otherwise, where
      $$h_2(x_1,x_2):=\EE\lt[f(x_1,x_2,X_3,\cdots,X_i)\rt]$$
  \item [$\bullet$]   $\cdots \ \cdots $

  \item [$\bullet$] At round $i$, $s^*_i=1$  if
  $$E_P\lt[f(x_1,\cdots,x_{i-1},W^L)-f(x_1,\cdots,x_{i-1},W^R)\rt]\Big|_{(x_1,\cdots,x_{i-1})=(Z_1^{s^*},\cdots,Z_{i-1}^{s^*})}\geq 0$$
      and $s^*_i=0$ if otherwise.
\end{enumerate}
It is easy to check that
$$\begin{aligned}
E_{P^{s^*}}\lt[f\lt(X_1,\cdots,X_i\rt)\rt]&=\EE\lt[\EE\lt[\cdots \EE\lt[f(x_1,\cdots,x_{i-1},X_i)\rt]\big|_{x_{i-1}=X_{i-1}}\cdots\rt]\Big|_{x_1=X_1}\rt]\\
&=\EE\lt[f(X_1,\cdots,X_i)\rt].
\end{aligned}$$
 Therefore
$$\sup_{s\in\cs}E_P\lt[f\lt(Z_1^s,\cdots,Z_i^s\rt)\rt]=\EE \lt[f\lt(X_1,\cdots,X_i\rt)\rt].$$

With the notation $\cg_i^s:=\sigma(Z_1^s,\cdots, Z_i^s),\ \cf_i:=\cb(\br^i),\ i\geq 1$ and $\cg_0^s:=\{\Omega_0,\emptyset\},\ \cf_0:=\{\br,\emptyset\}$, we notice that $E_P\lt[Z_i^s|\cg_{i-1}^s\rt] \stackrel{d}{=}E_{P^s}\lt[X_i|\cf_{i-1}\rt],\ i\geq 1$. Therefore, we conclude that for each $\vp\in C_{{Lip}}(\br)$,
\begin{equation}\label{TAB}
\begin{aligned}
&\lim_{n\to\infty}\sup_{s\in\cs}E_P\lt[\vp\lt(\frac{\sum_{i=1}^{n}Z_i^s-E_P\lt[Z_i^s|\cg_{i-1}^s\rt]}{\sqrt{n}}\rt)\rt]\\
= &\lim_{n\to\infty}\sup_{s\in\cs}E_{P^s}\lt[\vp\lt(\frac{\sum_{i=1}^{n}X_i-E_{P^s}\lt[X_i|\cf_{i-1}\rt]}{\sqrt{n}}\rt)\rt]\\
= &\lim_{n\to\infty}\sup_{P\in\cp}E_{P}\lt[\vp\lt(\frac{\sum_{i=1}^{n}X_i-E_P\lt[X_i|\cf_{i-1}\rt]}{\sqrt{n}}\rt)\rt] =\eg [\vp(\xi)],
\end{aligned}\end{equation}
where $\xi\sim\cn(0,[\ls^2,\os^2])$.

(\ref{TAB}) implies that the limit distribution of the normalized sum of rewards under all strategies can be characterized by the $G$-expectation corresponding to the distributions of random payoffs from both arms. 
The similar argument still holds for multiple-armed bandit problem.

\subsection{Examples}

The following example is motivated by Teran \cite{Teran}, which shows that CLT also holds without underlying probability measure.
\begin{example}\label{ex1}
Let $(\Omega,\cb(\Omega))$ be a measurable space with $\Omega=0\cup\bn$. Let $X_n(\omega)$ be the $n$-th bit in the binary representation of $\omega$, i.e.,
$$\omega=\sum_{n=1}^\infty 2^{n-1}X_n(\omega)$$
with $X_n(\omega)\in\{0,1\}$.

There does not exist any probability measure $P$ on $\Omega$ such that $\{X_i\}_{i\in\bn}$ is a classical i.i.d. sequence with $E_P[X_1]=\frac{1}{2}$ (see Guo et al. \cite{gll3}). But it is feasible to construct a proper linear expectation space $(\Omega,\mathcal{H},\be)$ such that $\{X_i\}_{i\in\bn}$ is i.i.d. under $\be$ with $\be[X_1]=\frac{1}{2}$.

In particular, by Corollary \ref{cltv}, we have
$$\be\lt[\vp\lt(\frac{\sum_{i=1}^n\lt(X_i-\frac{1}{2}\rt)}{\sqrt{n}}\rt)\rt]=\mathbb{E}_G[\vp(\xi)], \ \ \ \forall \vp\in C_{b.Lip}(\br),$$
where $\xi\sim\mathcal{N}(0,\frac{1}{4})$.

\end{example}

Indeed, for fixed $n\in\bn$, let $P^{(n)}$ be the probability measure on $\Omega$ such that $$P^{(n)}(\{i\})=\frac{1}{2^n},\ 0\leq i\leq 2^n-1, \ \ \  P^{(n)}(\{i\})=0, \ i\geq 2^n. $$
It is clear that $\{X_i\}_{1\leq i\leq n}$ is a classical i.i.d. sequence with $E_{P^{(n)}}[X_1]=\frac{1}{2}$.

 Now we construct the proper $(\Omega,\mathcal{H},\be)$.
 Let
 $$\mathcal{H}=\{\vp(X_1,\cdots,X_n), \forall n\in\bn, \forall \vp\in C_{Lip}(\br^n)\},$$
 the linear functional $\be$ on $\mathcal{H}$ is defined by
$$\be[\vp(X_1,\cdots,X_n)]:=E_{P^{(n)}}[\vp(X_1,\cdots,X_n)],\ \ \forall n\in\bn.$$
It is easily verified that $\be$ is well defined. We have
$$\be[\vp(X_i)]=\frac{1}{2}(\vp(0)+\vp(1)), \ \forall \vp\in C_{b.Lip}(\br), \ \forall i\in\bn,$$
and for each $n\geq 2$ and $\vp\in C_{b.Lip}(\br^n)$,
$$\be[\vp(X_1,\cdots,X_n)]=\be[\be[\vp(x_1,\cdots,x_{n-1},X_n)]_{(x_1,\cdots,x_{n-1})=(X_1,\cdots,X_{n-1})}].$$
We also note that, for fixed $\omega\in\Omega$, $X_n(\omega)\to 0$, but $\be[X_n]=\frac{1}{2}\nrightarrow 0$. Hence $\be$ is not regular. In fact, if the linear functional $\be$ satisfying (a) and (b) in Definition \ref{sub} is regular, then by the Daniell-Stone theorem, there must exist a probability measure represented by $\be$, which is a contradiction.

Next we consider the above example in sublinear expectation space.

\begin{example}
Let $\cp$ be the set of all probability measures on $(\Omega,\cb(\Omega))$ with $\Omega=0\cup\bn$ and $\be[\cdot]=\sup_{P\in\cp}E_P[\cdot]$. Then we have $\{X_i\}_{i\in\bn}$ is an i.i.d. sequence under $\be$ (see Guo et al. \cite{gll}).

It is clear that
$$\be[X_1]=1, \ \ -\be[-X_1]=0, \ \ \overline{V}(X_1)=\frac{1}{4}, \ \ \underline{V}(X_1)=0.$$

By Corollary \ref{cltv}, for each continuous function $\vp\in C(\br)$ with linear growth,
$$\lim_{n\rightarrow\infty}\sup_{P\in\cp}E_P\lt[\vp\lt(\frac{\sum_{i=1}^n(X_i-E_P[X_i|\cf_{i-1}])}{\sqrt{n}}\rt)\rt]=\eg[\vp(\xi)],$$
where $\xi\sim\mathcal{N}(0,[0,\frac{1}{4}])$.
\end{example}

The last example shows that condition (\ref{uc}) in CLT can not be weaken to $\be[|X_1|^2]<\infty$.
\begin{example}
Let $\Omega = \mathbb{Z}$, $\mathcal{F}=\mathcal{B}(\mathbb{Z})$, $\mathcal{P}=\{P_{k}, k \geq 1\}$, where $P_{k}(\{0\})=1-\frac{1}{k^2}$, $P_{k}(\{k\})=P_{k}(\{-k\})=\frac{1}{2k^2}$. Consider a function $X$ on $\mathbb{Z}$ defined by $X(\omega)=\omega,  \omega\in \mathbb{Z}$ and the sublinear expectation $\be[\cdot]=\sup_{P\in\cp}E_P[\cdot]$. We note that $\be[X]=\be[-X]=0$ and $\be[X^2]=-\be[-X^2]=1$. We are able to construct an i.i.d. sequence $\{X_{i}\}_{i\in\bn}$ on some sublinear expectation space $(\tilde{\Omega},\tilde{\mathcal{H}},\tilde{\mathbb{E}})$ such that $X_{i}$ has the same distribution as $X$. Then CLT does not hold.
\end{example}
In fact, we consider $\varphi(x)=1-|x|$. Let $S_n=X_1+\cdots+X_n$.

We firstly observe that
\begin{align*}\tilde{\mathbb{E}}\lt[\vp\lt(\frac{x+X}{\sqrt{n}}\rt)\rt]&=\sup_{k\in\mathbb{N}}\lt\{\lt(1-\frac{1}{k^2}\rt)\vp\lt(\frac{x}{\sqrt{n}}\rt)+\frac{1}{2k^2}\lt(\vp\lt(\frac{x-k}{\sqrt{n}}\rt)+\vp\lt(\frac{x+k}{\sqrt{n}}\rt)\rt)\rt\}\\
&=\vp\lt(\frac{x}{\sqrt{n}}\rt)+\frac{1}{\sqrt{n}}\sup_{k\in\mathbb{N}}\lt\{\frac{2|x|-|x+k|-|x-k|}{2k^2}\rt\}=\vp\lt(\frac{x}{\sqrt{n}}\rt).
\end{align*}
Then
\begin{align*}
\tilde{\mathbb{E}}\lt[\vp\lt(\frac{S_n}{\sqrt{n}}\rt)\rt]&=\tilde{\mathbb{E}}\lt[\tilde{\mathbb{E}}\lt[\vp\lt(\frac{x+X_n}{\sqrt{n}}\rt)\rt]|_{x=X_1+\cdots+X_{n-1}}\rt]
=\tilde{\mathbb{E}}\lt[\vp\lt(\frac{S_{n-1}}{\sqrt{n}}\rt)\rt]\\
&=\cdots=\tilde{\mathbb{E}}\lt[\vp\lt(\frac{X_1}{\sqrt{n}}\rt)\rt]=\vp(0)=1.
\end{align*}

Finally,
$$\lim_{n\rightarrow\infty}\tilde{\mathbb{E}}\lt[\varphi\lt(\frac{S_n}{\sqrt{n}}\rt)\rt]=1>\eg[\vp(\xi)],$$
where $\xi\sim\mathcal{N}(0,1)$.

\section*{Appendix: Law of large numbers under sublinear expectation}
In this section, we firstly consider the law of large numbers on canonical space.

\begin{theorem}
Let $\{X_i\}_{i\in\bn}$ be an independent canonical sequence on $(\br^\bn,\cb(\br^\bn), \cp)$, where $\cp$ is constructed by (\ref{e1}) with given convex and weakly compact sets $\{\cp_i\}_{i\in\bn}$. Assume that $\{X_i\}_{i\in\bn}$ have the same mean-uncertainty, i.e.
$$\ou=\EE[X_n],\ \ \lu=-\EE[-X_n], \ \forall n\in\bn,$$
and satisfy
\begin{equation}\label{uccc}
\lim_{\lambda\to\infty}\sup_{i}\mathbb{E}^{\cp}[(|X_i|-\lambda)^+]=0.
\end{equation}
Then for each $\vp\in C(\br)$ with linear growth, we have
\begin{equation}\label{lln1}
\lim_{n\to\infty}\pe\lt[\vp\lt(\frac{\sum_{i=1}^nX_i}{n}\rt)\rt]=\max_{\lu\leq\mu\leq\ou}\vp(\mu).
\end{equation}

In particular,
\begin{equation}\label{lln2}
\lim_{n\to\infty}\pe\lt[d_\Theta\lt(\frac{\sum_{i=1}^nX_i}{n}\rt)\rt]=0,
\end{equation}
where $d_\Theta(x)=\inf_{y\in\Theta}|x-y|$ with $\Theta=[\lu,\ou]$.
\end{theorem}

\begin{proof}
It is clear that (\ref{lln2}) is implied from (\ref{lln1}) by taking $\vp(x)=d_\Theta(x)$. To prove (\ref{lln1}), we only need to prove it for $\vp\in C_{b,Lip}(\br)$. Indeed, since
\begin{align*}
\pe\lt[\lt(\lt|\frac{S_n}{n}\rt|-\lambda\rt)^+\rt]\leq\frac{1}{n}\pe[(\sum_{i=1}^n|X_i|-\lambda)^+]\leq\sup_i\pe[(|X_i|-\lambda)^+]\to 0,
\end{align*}
we can extend (\ref{lln1}) from $C_{b,Lip}(\br)$ to $C(\br)$ with linear growth by the similar argument in Peng \cite{P2019}.

The proof of (\ref{lln1}) for $\vp\in C_{b.Lip}(\br)$ is divided into two steps.

Firstly, let $\tX_k=(-n\vee X_k)\wedge n$, $1\leq k\leq n$, $\ou_n=\pe[\tX_n]$ and $\lu_n=-\pe[-\tX_n]$, $\forall n\in\bn$. We denote $\{\cf_i\}_{i\in\bn}$ the natural filtration generalized by $\{X_i\}_{i\in\bn}$ and  $S_n=\sum_{i=1}^nX_i$, $\tilde{S}_n=\sum_{i=1}^n\tX_i$.

For each $P\in\cp$, we have
\begin{align*}
&E_P\lt[\lt|\vp\lt(\frac{S_n}{n}\rt)-\vp\lt(\frac{\tilde{S}_n}{n}\rt)\rt|\rt]\\
\leq &\frac{L_\vp}{n}E_P\left[\left|\sum_{i=1}^n(X_i-\tX_i)\right|\right]\\
\leq& L_\vp\sup_i\pe[(|X_i|-n)^+]\to 0,
\end{align*}
and
\begin{align*}
&E_P\lt[\lt|\vp\lt(\frac{\sum_{i=1}^n\tilde{S}_n}{n}\rt)-\vp\lt(\frac{\sum_{i=1}^nE_P[\tX_i|\cf_{i-1}]}{n}\rt)\rt|\rt]\\
\leq& \frac{L_\vp}{n}E_P\lt[\lt|\sum_{i=1}^n(\tX_i-E_P[\tX_i|\cf_{i-1}])\rt|\rt]\\
\leq&\frac{L_\vp\sqrt{\sum_{i=1}^nE_P[(\tX_i-E_P[\tX_i|\cf_{i-1}])^2]}}{n}\\
\leq &\frac{2L_\vp\sqrt{\sum_{i=1}^nE_P[\tX_i^2]}}{n}\to 0
\end{align*}
where $L_\vp$ is the Lipschitz constant of $\vp$. The proof of the last convergence is very similar to the classical case.

It is clear that $\ou_n\to\ou$ and $\lu_n \to \lu$ by (\ref{uccc}).
Combining with Proposition 2.4 in Guo et al. \cite{GL},  we imply that for each $\ve>0$, there exists $N$ large enough such that for every $n>N$,
$$\vp\lt(\frac{\sum_{i=1}^nE_P[\tX_i|\cf_{i-1}]}{n}\rt)\leq\max_{\frac{\sum_{i=1}^n\lu_i}{n}\leq\mu\leq\frac{\sum_{i=1}^n\ou_i}{n}}\vp(\mu) \to\max_{\lu\leq\mu\leq\ou}\vp(\mu).
$$

Thus we can conclude that
$$\limsup_{n\to\infty}\pe\lt[\vp\lt(\frac{S_n}{n}\rt)\rt]\leq\max_{\lu\leq\mu\leq\ou}\vp(\mu).$$

Secondly, for fixed $\vp\in C_{b.Lip}(\br)$, there exists $\mu^*\in[\lu,\ou]$ such that $$\vp(\mu^*)=\max_{\lu\leq\mu\leq\ou}\vp(\mu).$$
There exists $P_*\in\cp$ such that $\{X_i\}_{i\in\bn}$ is an independent sequence with
$E_{P^*}[X_i]=\mu^*, \ \forall i\in\bn.$
 The classical law of large numbers shows that
$$\lim_{n\to\infty}E_{P^*}\lt[\vp\lt(\frac{S_n}{n}\rt)\rt]=\vp(\mu^*),$$
which implies that
$$\liminf_{n\to\infty}\pe\lt[\vp\lt(\frac{S_n}{n}\rt)\rt]\geq\max_{\lu\leq\mu\leq\ou}\vp(\mu).$$
The proof is completed.

\end{proof}

By Remark \ref{r48}, we extend the law of large numbers from canonical space to the sublinear expectation space, which slightly improves the law of large numbers in Peng \cite{P2019}.
\begin{theorem}
Let $\{X_i\}_{i\in\bn}$ be an independent sequence on sublinear expectation space $(\Omega,\mathcal{H},\be)$ with same mean-uncertainty and $$\lim_{\lambda\rightarrow\infty}\sup_i\be[(|X_i|-\lambda)^+]=0$$

For each $\vp\in C(\br)$ with linear growth, we have
$$\lim_{n\to\infty}\be\lt[\vp\lt(\frac{S_n}{n}\rt)\rt]=\max_{\lu\leq\mu\leq\ou}\vp(\mu),$$
where
$\lu=-\be[-X_1]$ and $\ou=\be[X_1]$.
\end{theorem}
%
%

\section*{Acknowledgments}
The author gratefully acknowledges the incredibly helpful suggestions from Prof. Shige PENG during the preparation of the paper.

This work was supported by NSF of Shandong Provence (No.ZR2021MA018),  NSF of China (No.11601281),  National Key R\&D Program of China (No.2018YFA0703900) and the Qilu Young Scholars Program of Shandong University.

\end{document}